\documentclass[twoside, 12pt]{article}
\usepackage{amssymb, amsmath, mathrsfs, amsthm}
\usepackage{graphicx}
\usepackage{color}
\usepackage[top=1in, bottom=1in, left=1in, right=1in]{geometry}
\usepackage{float, caption, subcaption}
\usepackage{url}
\usepackage{hyperref}

\newtheorem{theorem}{Theorem}[section]
\newtheorem{lemma}[theorem]{Lemma}

\newtheorem{corollary}[theorem]{Corollary}

\newtheorem{proposition}{Proposition}[section]

\newcommand{\R}{\mathbb{R}}
\newcommand{\Q}{\mathbb{Q}}
\newcommand{\N}{\mathbb{N}}
\newcommand{\B}{\mathcal{B}}
\newcommand{\Sp}{\mathcal{S}}
\newcommand{\G}[1][2]{G^1_{\R^#1}}
\newcommand{\del}{\partial}

\DeclareMathOperator{\Conv}{Conv}
\DeclareMathOperator{\Cone}{Cone}
\DeclareMathOperator{\diam}{diam}

\usepackage{mathtools}
\DeclarePairedDelimiter{\ceil}{\lceil}{\rceil}

\begin{document}
\title{\bf Connectedness of Unit Distance Subgraphs Induced by Closed Convex Sets}
\author{Remie Janssen\footnote{Delft Institute of Applied Mathematics, Delft University of Technology,
Delft, The Netherlands. Research funded by the Netherlands Organization for Scientific Research (NWO), Vidi grant 639.072.602 of dr. Leo van Iersel.} \\ \href{mailto:remiejanssen@gmail.com}{remiejanssen@gmail.com}
\and Leonie van Steijn \footnote{Mathematical Institute, Leiden University,  
Leiden, The Netherlands.}\\ \href{mailto:l.van.steijn@math.leidenuniv.nl}{l.van.steijn@math.leidenuniv.nl}}
\maketitle

\begin{abstract}
    The unit distance graph $\G[d]$ is the infinite graph whose nodes are points in $\R^d$, with an edge between two points if the Euclidean distance between these points is $1$. The 2-dimensional version $\G$ of this graph is typically studied for its chromatic number, as in the Hadwiger-Nelson problem. However, other properties of unit distance graphs are rarely studied. Here, we consider the restriction of $\G[d]$ to closed convex subsets $X$ of $\R^d$. We show that the graph $\G[d][X]$ is connected precisely when the radius of $r(X)$ of $X$ is equal to $0$, or when $r(X)\geq 1$ and the affine dimension of $X$ is at least $2$. For hyperrectangles, we give bounds for the graph diameter in the critical case that the radius is exactly 1.
\end{abstract}

\thispagestyle{empty}

\section{Introduction}
In the Hadwiger-Nelson problem, the aim is to colour the plane with as few colours as possible, so that no pair of points at distance 1 from each other have the same colour. Equivalently, it asks to find the chromatic number of $\G[2]$, the unit distance graph of the plane which has node set $\R^2$ and an edge between two points $u,v\in\R^2$ precisely when they are at distance 1 from each other. This problem is typically tackled by studying finite subgraphs of $\G[2]$. 
However, $\G[2]$ has an interesting structure when restricted to (infinite) connected subsets of the plane as well, such as a strip of the plane $S_r=\R\times [0,r]$. It turns out that, in the induced subgraph $\G[2][S_r]$ of $\G[2]$, the existence of cycles of a given length and the chromatic number both depend on the size $r$ of the strip \cite{bauslaugh1998tearing,kruskal2008chromatic,axenovich2014chromatic}. A generalisation of these results can allegedly be found in \cite{kanel2018chromatic} (which is written in Russian, a language neither of the authors can read).

In this paper, we study subgraphs of $\G[d]$ (the unit distance graph of $\R^d$) induced by closed and convex subsets of $\R^d$. However, instead of studying the chromatic number of these graphs, we investigate their connectedness and graph diameters. In other words, we aim to characterize when there is a path between any two points in a closed convex set using unit distance steps. Additionally, we study bounds for the maximal number of unit steps needed to go from any point to any other point in a hyperrectangle.

\section{Preliminaries}

\subsection{Subsets of $\R^d$}
First, we define several basic subsets of $\R^d$ that we will use throughout this paper. We will use $d(x,y)$ to denote the Euclidean distance, and $x\cdot y$ for the standard inner product between $x,y\in\R^d$.

A ball is defined by its center and its radius: Let $d\in\N$, $v\in\R^d$, and $r\in\R$, then $\B^d(v,r)=\{x\in\R^d: d(x,v)\leq r\}$ is the $d$-dimensional ball of radius $r$ around $v$. The sphere $\Sp^{d-1}(v,r):=\{x\in\R^d: d(x,v)= r\}$ is the boundary $\del \B^d(v,r)$ of the ball $\B^d(v,r)$.

The $d$-dimensional hypercube with side lengths $l$ is denoted $C^d(l)$. We always use the embedding $\{x\in\R^d: x_i\in [0,l]\}$ of $C^d(l)$ in $\R^d$.
Similarly, a $d$-dimensional hyperrectangle is defined by a $d$-tuple of non-negative real numbers $l=(l_1,\ldots, l_d)\in\R^d$. We denote this hyperrectangle by $R^d(l)=\{x\in\R^d: x_i\in[0,l_i]\}$.

Let $P,Q\in\R^d$, then $PQ$ denotes the line segment between $P$ and $Q$. Moreover, we write $H(P,Q)=\{x\in\R^d:(Q-P)\cdot(x-P)\leq 0\}$, which is the closed half-space bounded by the hyperplane through $P$ perpendicular to $PQ$ that does not include $Q$. 

All of these subsets of $\R^d$, except for the sphere, are convex. In other words, for any pair of points in such a set (i.e., a ball, a hypercube, a hyperrectangle, a line segment, a hyperplane, or a halfspace) the line segment between them is also contained in the set. 

To construct convex sets, we use the convex hull $\Conv(A)$ of a set $A\subset\R^d$, which is the smallest convex set containing $A$. The cone $\Cone(A)$ of a set of points $A$ is the set $\{\lambda x: x\in\Conv(A), \lambda\geq0\}$. For a finite set of points $A=(a_1,\ldots,a_n)$ in $\R^d$ the convex hull is $\{\lambda_1a_1+\cdots+\lambda_na_n:\sum_{i=1}^n\lambda_i=1\}$. If all $a_i$ are affinely independent, then $S=\Conv(A)$ is an $n$-simplex. For each subset $I\subseteq [n]$, the set $\Conv(\{a_i\}_{i\in I})$ is an $|I|$-dimensional face of $\Conv(A)$. 

Finally, to construct additional subsets of $R^d$, we use the Minkowski sum and scaling of a set. We denote the Minkowski sum by $\oplus$, i.e., $X\oplus Y := \{x+y: x\in X, y\in Y\}$. Scaling a set $X\subseteq \R^d$ by a factor $\lambda\in\R_{\geq 0}$ is defined as $\lambda X := \{\lambda x\in \R^d: x\in X\}$. Note that the Minkowski sum of two convex sets is convex, and $\lambda X$ is convex iff $X$ is convex.

\subsection{Enclosing balls}
Let $X\subseteq \R^d$, then any ball $\B^d(v,r)\supseteq X$ containing $X$ is an enclosing ball of $X$. For each bounded set $X\subset\R^d$, there is a unique minimal enclosing ball (m.e.b.), which is the enclosing ball of $X$ with minimal radius.
The radius $r(X)$ of $X$ is defined as the radius of the m.e.b.~of $X$.

Note that the center of the m.e.b.~of a convex set lies in the convex set. A related lemma is the following, which, according to \cite{fischer2003fast}, is well-known and can be traced back to Seidel (no reference given). 

\begin{lemma}[Seidel]\label{lem:Seidel}
Let $X$ be a set of points on the boundary of some ball $B$ with center $C$. Then $B$ is the m.e.b.~of $X$ if and only if $C\in \Conv(X)$.
\end{lemma}

For a simplex $S\subseteq \R^d$ the center of the m.e.b of $S$ is always contained in $S$. This center may lie on a face of $S$. Here, we are interested in so-called well-centered simplices, which are simplices where this is not the case. In other words, a simplex $S$ is well-centered if the center of its m.e.b.~is contained in its relative interior \cite{vanderzee2010well,vanderzee2013geometric}. The relative interior of $S$ is the interior of $S$ when viewed as a subset of the affine span of $S$. 

\subsection{Graphs}
We now define the main subject of this paper, the unit distance graph of a subset of $\R^d$. The unit distance graph in dimension $d$ is denoted $\G[d]$; it is the infinite graph whose vertex set is $\R^d$, and two points $x,y\in\R^d$ are connected by an edge if $d(x,y)=1$.

We restrict this graph to subsets of $\R^d$ by considering induced subgraphs. Let $G=(V,E)$ be a graph, and $V'\subseteq V$ a subset of the vertices, then $G[V']$ is the subgraph of $G$ induced by $V'$: the graph with vertices $V'$ and an edge between $u,v\in V'$ if $\{u,v\}\in E$.

The graphs we are interested in are $\G[d][X]$, where $X\subset \R^d$ is compact and convex. In particular, we will study the connectedness and the diameter of these graphs, which are invariant under translation and rotation of $X$. The diameter $\diam(G)$ of a graph $G$ is defined as the maximal graph distance between any pair of vertices of $G$. The graph distance between two nodes $x$ and $y$ of $G$ is the minimum number of edges in a path between $x$ and $y$ through $G$. To prevent confusion between the Euclidean distance and the graph distance, we will henceforth reserve the term distance for the Euclidean distance, and the term diameter for the graph diameter. The graph distance will always be given in a number of steps. If there is a sequence of (unit distance) steps between two points $u$ and $v$, then we say that $v$ can be reached from $u$.

\section{Triangles and $n$-Simplices}

\subsection{Wiggling Through Triangles}

The following lemmas shows that each sufficiently large triangle has enough `wiggle room' to reach a significant part of one of the sides of the triangle using steps of length exactly one. This will be instrumental when proving that $\G[d][S]$ is connected for an $n$-simplex $S$.

\begin{lemma}\label{lem:RectangleWiggle}
Let $R=[0,1+x]\times[0,h]\subseteq \R^2$ for some $h>0$ and $0\leq x\leq 1$. Then, for each pair of nodes $u,v\in\G[2][R]$ that lie on the line segment $[0,x]\times\{0\}$, there is a path between $u$ and $v$ in $\G[2][R]$. 

Moreover, this path consists of at most 2 steps if $h\geq 1$, and at most $4 \ceil*{\frac{x}{2(1-\sqrt{1-h^2})}}$ steps if $h<1$.
\end{lemma}
\begin{proof}
To prove the lemma, we start in a point $P\in[0,x]\times\{0\}$ and show that we can reach all points on a small line segment to the right of $P$ in four steps. Obviously, if $h\geq 1$, each point on $[0,x+1]\times\{0\}$ can be reached in at most two steps from $P$, so we assume $h<1$. 

\begin{figure}[t]
    \centering
    \includegraphics[width=.6\textwidth]{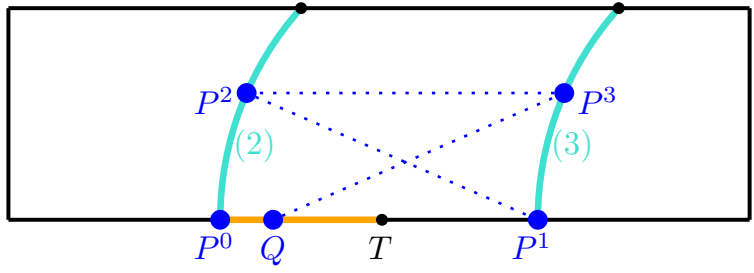}
    \caption{Wiggling through a rectangle from $P^0=(p,0)$ to $Q=(q,0)$ in Lemma~\ref{lem:RectangleWiggle}. The point $Q$ must lie on the line segment $P^0T$ (orange), where $T=(p+2(1-\sqrt{1-h^2}),0)$. Each point on $P^0T$ can be reached from some point on the circle arc (3), which is simply the circle arc (2) translated to the right by distance 1. Moreover, (2) is the circle arc with points at distance one from the point $P^1=(p+1,0)$. Hence, we can reach $Q$ from $P^0$ via $P^1$, some point $P^2$ on (2), and a point $P^3$ on (3). All dotted lines have unit length.}
    \label{fig:WiggleRectangle}
\end{figure}

Let $P=P^0=(p,0)$ with $0\leq p\leq x$, and let $Q=(p+q,0)$ such that $p+q/2\leq x$ and $0\leq q\leq 2(1-\sqrt{1-h^2})\leq 2$. Then, by taking the following four unit steps, we can reach the point $Q$ from $P^0$ (Figure~\ref{fig:WiggleRectangle}). First, go a unit step to the right to $P^1=(p+1,0)$. Next, go back to the left and up, to the point \[P^2=\left(p+q/2,\sqrt{1-(1-q/2)^2}\right).\] Note that $P^2\in \Sp^2(P^1,1)\cap R$. Indeed, $P^2\in \Sp^2(P^1,1)$ because $d(P^1,P^2)=1$, and $P^2\in R$ because $0\leq p+q/2\leq x$ and \[\sqrt{1-(1-q/2)^2}\leq \sqrt{1-(1-(1-\sqrt{1-h^2}))^2}= h.\] Take a third unit step to the right, to $P^3=P^2+(1,0)$, which lies in $R$ because $p+q/2+1\leq x+1$. Lastly, take a unit step from $P^3$ to $Q$, where the fact that $d(P^3,Q)=1$ follows from a simple calculation. 

Now let $u=(u_0,0)$ and $v=(v_0,0)$ be two points on the line segment $[0,x]\times\{0\}$ such that $u_0\leq v_0$. By repeatedly using the sequence of four steps given above, we can reach $v$ from $u$ in $\G[2][R]$. This takes at most \[4\ceil*{\frac{v_0-u_0}{2\left(1-\sqrt{1-h^2}\right)}}\] 
steps if $h\leq 1$ and at most $2$ steps otherwise. Finally, note that we can therefore upper bound the number of steps between $u$ and $v$ by 2 if $h\geq 1$ and, otherwise, by
\[4\ceil*{\frac{x}{2\left(1-\sqrt{1-h^2}\right)}}.\]
\end{proof}

\begin{lemma}\label{lem:TriangleWiggle}
Let $T\in \R^d$ be the triangle $\Conv(\{P^0,P^1,P^2\})$. Assume that $d(P^0,P^1)=1+x>1$ and that there is a point $A$ in the relative interior of $T$ such that $d(P^0,A)=d(P^1,A)=1$. Then for each pair of points $u,v\in P^0P^1\cap \B^d(P^0,x)$, there is a path from $u$ to $v$ in $\G[d][T]$.
\end{lemma}
\begin{proof}
Without loss of generality, we assume $d=2$, $P^0=(0,0)$, $P^1=(1+x,0)$, and $P^2=(P^2_1,P^2_2)$ with $P^2_2>0$. Finally, to simplify notation in the proof, we assume $0\leq P^2_1\leq 1+x$. Note that this does not affect the conclusion: if $P^2_1$ lies outside this region, we can restrict to the triangle $T' = T\cap [0,1+x]\times\R$. Indeed the result for $T$ then follows because the assumptions of the lemma still hold for $T'$, and $\G[d][T']$ is a subgraph of $\G[d][T]$ containing all $u,v\in P^0P^1\cap \B^d(P^0,x)$.

We first prove that we can reach some small neighbourhoods of $P^0$ and $P^1$ in $T$ using unit distance steps (Figure~\ref{fig:WiggleTriangle}). Then, we extend the small region around $P^0$ to the right by `wiggling'.  Finally, translating the neighbourhood of $P^1$ to the left by a distance of one, we reach the desired part of the line segment $P^0P^1$.

\begin{figure}[ht!]
    \centering
    \includegraphics[width=.9\textwidth]{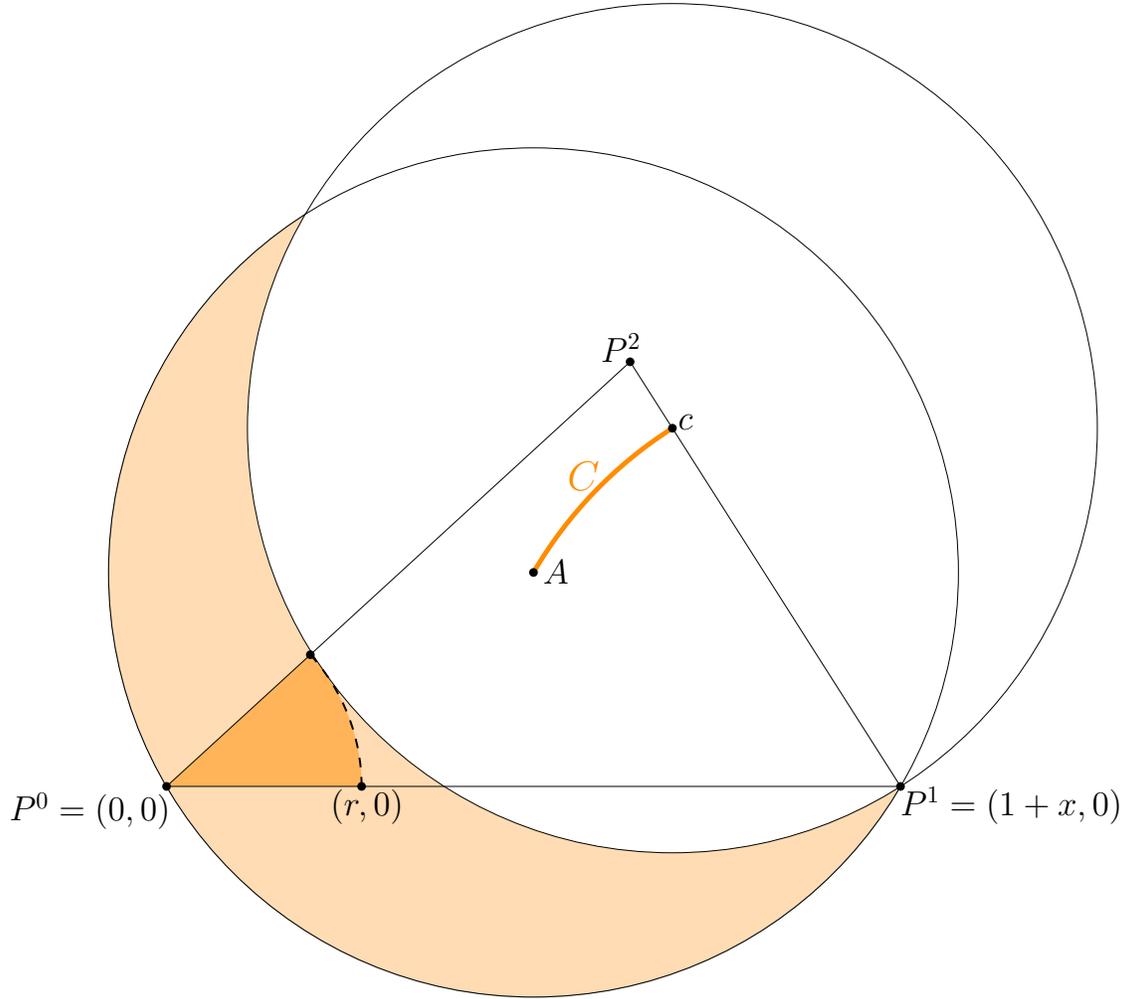}
    \caption{Reaching the corners of a triangle as in Lemma~\ref{lem:TriangleWiggle}. To reach  any point $Q$ in the shaded are within $T$ from $A$, take one step to $P^1$, then one to a point on the arc $C$, and finally to $Q$. There exists a suitable point on $C$ because the shaded area lies in $\B^2(A,1)\setminus \B^2(c,1)$ and $C$ is a continuous path from $c$ to $A$.}
    \label{fig:WiggleTriangle}
\end{figure}
 
Note that there is a small arc of the circle $\Sp^1(P^1,1)$ containing $A$ that lies within $T$, namely $C=\Sp^1(P^1,1)\cap \B^2(A,\epsilon)$ for some $\epsilon>0$ sufficiently small. Note that the endpoint $c$ of $C$ to the right of $A$ has distance $d(c,P^0)=1+r>1$ to $P^0$, so $d(Q,c)\geq 1$ for all $Q\in \B^2(P^0,r)\cap T$ by the triangle inequality. Moreover, because $P^2_1\geq 0$, we have $d(Q,A)\leq 1$ for all $Q\in \B^2(P^0,r)\cap T$. These two distance bounds, together with the continuity of the circle arc $C$ from $c$ to $A$ and the intermediate value theorem imply that each $Q\in \B^2(P_0,r)\cap T$ is at distance 1 from some point on $C$. By symmetry in $P^0$ and $P^1$, there is some $0<r'\in \R$ such that each $Q'\in \B^2(P^1,r')\cap T$ can be reached from $A$ as well.

Note that we now know we can reach $[0,r]\times\{0\}$ and $[1+x-r',1+x]\times\{0\}$ from $A$. To see that we can also reach $[x-r',x]\times\{0\}$, we note that this is just the line segment $[1+x-r',1+x]\times\{0\}$ translated to the left by distance 1. The only remaining part of $[0,x]\times\{0\}$ we need to reach is $[r,x-r']\times\{0\}$. If $r+r'\geq x$, then we are done, so suppose $r+r'<x$. The required result follows immediately from Lemma~\ref{lem:RectangleWiggle} when we observe that there is a rectangle $R=[r,1+x-r']\times[h]\subseteq T$ for some $h>0$, whose base has width $w=1+x-r'-r$. Indeed, the lemma then implies that we can reach $[r,r+(w-1)]\times\{0\}$, with $r+(w-1)=x-r'$.
\end{proof}

In this proof, we have used the existence of small balls, but these balls may actually be reasonably large, so that the steps may make significant headway along $P_0P_1$. However, when the angle at $P_2$ is very large, this wiggle space may be limited. 

\begin{lemma}\label{lem:ObtuseTriangleConnected}
Let $T\subseteq \R^2$ be an obtuse triangle with radius $r(T)=1$, then $\G[2][T]$ is connected.
\end{lemma}
\begin{proof}
As $T$ is obtuse, the radius of $T$ is half of the longest side of $T$. Hence, without loss of generality, let the triangle be the convex hull of the points $(0,0)$, $(2,0)$ and $P=(P_1,P_2)$ with $P_1,P_2>0$. We prove that there is a sequence of steps between any two points in base $B$ of the triangle, i.e., the line segment between $(0,0)$ and $(2,0)$. The result then simply follows from the fact that each point in the triangle is at distance one from some point on $B$. To see this, let $Q=(Q_1,Q_2)\in T$ be arbitrary, and assume without loss of generality that $Q_1\leq 1$, so that we have $d(Q,(2,0))\geq 1$. As $(1,0)$ is the center of the m.e.b.~of $T$, we also have $d(Q,(1,0))\leq 1$. Applying the intermediate value theorem on the function $d(Q,\cdot)$ along the line segment $(0,1)(0,2)$, we can conclude that there is a point $b$ on $B$ such that $d(b,Q)=1$.

To see that there is a sequence of steps between any two points in base $B$ of the triangle we use a technique similar to the proof of Lemma~\ref{lem:TriangleWiggle}. We first note that we can reach some intervals $[0,r]\times\{0\}$ and $[2-r',2]\times\{0\}$ from $(1,0)$ by going via the arcs above $(1,0)$ of the circles around $(0,0)$ and $(2,0)$ of radius one (similar to Figure~\ref{fig:WiggleTriangle}, but with $A$ on $B$ instead of in the relative interior of $T$). Then, by translating these line segments to the left and the right by a distance of one, we can also reach the line segment $[1-r',1+r]\times\{0\}$. The remaining parts of $B$ are $[r,1-r']\times\{0\}$ and $[1+r,2-r']\times\{0\}$, which can both be reached through a rectangle $[r,2-r']\times[0,h]$ for some $h>0$ (Lemma~\ref{lem:RectangleWiggle}).
\end{proof}

\subsection{Connectedness of Simplices}
In this subsection, we will prove that $\G[d][S]$ is connected for any simplex $S$ of dimension at least $2$ with radius $r(S)=1$. To do this, we first show that each point in $S$ is at distance 1 from some point on a line segment $CP\subseteq S$ where $P$ is a vertex of $S$ with $d(C,P)=1$. Then, we show that we can reach each point on such a line segment from $C$ taking only unit steps. Together, these facts prove that $\G[d][S]$ is connected.

\begin{lemma}\label{lem:radiusToHalfSpace}
Let $S\subseteq\R^d$ be a simplex with m.e.b.~$B=\B^d(C,1)$ and let $P$ be a vertex of $S$ with $d(C,P)=1$. Then, each point in $H(C,P)\cap B$ is at distance one from some point on $CP$.
\end{lemma}
\begin{proof}
Let $Q\in H(C,P)\cap B$ be arbitrary. As $Q\in H(C,P)$ and $\B^d(P,1)\cap H(C,P)=\{C\}$, we have $d(P,Q)\geq 1$. Furthermore, because $Q\in B=\B^d(C,1)$ we have $d(C,Q)\leq 1$. Combining this with continuity of the function $d(Q,\cdot)$ along the line segment $CP$, the intermediate value theorem implies the existence of a point $U$ on $CP$ such that $d(Q,U)=1$.
\end{proof}

\begin{lemma}\label{lem:AllHalfSpacesCoverEverything}
Let $S\subseteq\R^d$ be a simplex with m.e.b.~$B=\B^d(C,r)$ and let $P^1,\ldots,P^n$ be the vertices of $S$ where $d(C,P^i)=r$, then $\bigcup_{i\in[n]}H(C,P^i)=\R^d$.
\end{lemma}
\begin{proof}
Without loss of generality, assume that $C=(0,\ldots,0)$. Suppose there exists a point $Q\in\R^d$ such that $Q\not\in H(C,P^i)$ for all $i\in [n]$. This implies that $Q\cdot P^i>0$ for all $i\in [n]$, so $P^i\in \R^d\setminus H(C,Q)$ for all $i\in [n]$. In other words, all $P^i$ lie strictly on one side of the hyperplane perpendicular to $CQ$ through $C$. Then, by Lemma~\ref{lem:Seidel}, the center $C$ must satisfy $C\in\Conv(\{P^1\ldots,P^n\})\subseteq \R^d\setminus H(C,Q)$. This contradicts the fact that $C\in H(C,Q)$. We conclude that there is no $Q\in\R^d$ such that $Q\not\in H(C,P^i)$ for all $i\in [n]$; in other words, each $Q\in\R^d$ lies in $H(C,P^i)$ for some $i\in [n]$, so $\bigcup_{i\in[n]}H(C,P^i)=\R^d$.
\end{proof}

\begin{lemma}\label{lem:WellCenteredSimplex}
Let $S\subseteq\R^d$ be a well-centered simplex of dimension at least $2$ with radius $r(S)=1$, then $\G[d][S]$ is connected.
\end{lemma}
\begin{proof}
Let $B=\B^d(C,1)$ be the m.e.b.~of $S=\Conv(\{V^1,\ldots,V^n\})$, 
and let $P$ be an arbitrary vertex of $S$. Because $S$ is well-centered, we have $d(C,V^i)=1$ for all $i$ and therefore, by Lemma~\ref{lem:AllHalfSpacesCoverEverything}, $H(C,P)\cap \{V^1,\ldots,V^n\}\neq\emptyset$. 
Hence, $S$ has a vertex $Q\in H(C,P)\cap \{V^1,\ldots,V^n\}$, and the 1-face $PQ$ has length at least $\sqrt{2}$.
Because $S$ is well-centered, $C$ is in the relative interior of $S$, so there is a point $R\in S$ such that the triangle $T=\Conv(\{P,Q,R\})$ contains the center $C$ of the m.e.b.~of $S$. 

Let $X$ be the point at distance $|PQ|-1$ from $Q$ on $PQ$. 
Applying Lemma~\ref{lem:TriangleWiggle} with $P^0=Q$, $P^1=P$, $P^2=R$ and $A=C$, we see that any point on $QX$ can be reached from $C$ through $\G[d][T]$.
Furthermore, there is a continuous path (circle arc) from $X$ to $C$ through $\Sp^{d-1}(P,1)\cap T$. 
All of the points on this path are at distance $1$ from $P$, so they can be reached from $C$ via $P$ in $\G[d][T]$. The union of this circle arc and the line segment $QX$ forms a continuous path $\mathcal{P}(Q,C)$ from $Q$ to $C$ through $T$, and each of the points on this path can be reached from $C$ in $\G[d][T]$. 

Now let $x\in CP$ be arbitrary. Using continuity of $d(x,\cdot)$ along $\mathcal{P}(Q,C)$ and the fact that $d(Q,x)\geq 1$ and $d(C,x)\leq 1$, the intermediate value theorem implies that there is a point $y\in\mathcal{P}(Q,C)$ with $d(x,y)=1$, so $x$ can be reached from $C$ in $\G[d][T]$ via y. As we chose $x\in CP$ arbitrarily, this shows that each point on $CP$ can be reached from $C$ in $\G[d][T]$. Moreover, as we chose an arbitrary vertex $P$ of $S$, we conclude that, for each vertex $V^i$ of $S$, all points on $CV^i$ can be reached from $C$ in $\G[d][S]$.

Now note that each point in $H(C,V^i)\cap B$ is at distance one from some point on $CV^i$, and $\bigcup_{i\in [n]}H(C,V^i)\cap B = B$ (Lemma~\ref{lem:AllHalfSpacesCoverEverything}). Finally, as $S\subseteq B$, each point in $S$ can be reached from $C$ in $\G[d][S]$. In other words, $\G[d][S]$ is connected.
\end{proof}

\begin{proposition}\label{prop:SimplexConnected}
Let $S\subseteq \R^d$ be a simplex with radius $r(S)=1$ of dimension at least $2$, then $\G[d][S]$ is connected.
\end{proposition}
\begin{proof}
Let $C$ be the center of the m.e.b.~$B$ of $S$, and note that $C$ lies in the relative interior of a face $F$ of $S$ of dimension at least 1. Moreover, the m.e.b.~of $F$ is $B$. 
If $F$ has dimension one, then there exists a point $R\in S\setminus F$ because $S$ has dimension at least $2$. Hence, by choosing $R$ close to $F$, $S$ contains the obtuse triangle $T=\Conv(\{R\}\cup F)$, and $\G[d][T]$ is connected by Lemma~\ref{lem:ObtuseTriangleConnected}. In particular, $F$ can be reached in its entirety from $C$ in $\G[d][S]$.

If $F$ has dimension at least 2, then $\G[d][F]$ is connected, because $F$ is a well-centered simplex of radius $1$ (Lemma~\ref{lem:WellCenteredSimplex}). 
In both cases, $\G[d][F]$ is connected and each point in $B$ is at distance one from some point on $F$. Hence, we conclude that $\G[d][S]$ is connected.
\end{proof}

The following lemma can improve the diameter bound resulting from the proofs above. It states that in a well-centered simplex $\Conv(\{V^1, \ldots, V^{d+1}\})$ of dimension $d$, each point on a line segment $CV^i$ is at most $d$ steps away from a point on $CV^j$ for each $j\neq i$. Hence, its consequence is that, during a sequence of unit steps between any two points in a well-centered simplex, we only have to `wiggle' once on one of the edges (i.e., apply the sequence of steps from Lemma~\ref{fig:WiggleTriangle} at most once). As the wiggling part is the part that may take the most steps, this could make the diameter bound a lot smaller.

\begin{lemma}\label{lem:OneRadiusThenAllRadii}
Let $S\subseteq\R^d$ be a simplex with m.e.b.~$\B^d(C,1)$,
and $P^1,\ldots,P^n$ the vertices of $S$ with $d(C,P^i)=1$. Then, for each $i\in[n]$ and $u\in CP^i$, there exists $v\in CP^1$ such that there is a $uv$-path of at most $n-1$ steps in $\G[d][S]$.
\end{lemma}
\begin{proof}
By Lemma~\ref{lem:radiusToHalfSpace}, we can reach each point on $CP^j$ from some point on $CP^i$ if $P^j\in H(C,P^i)$. Hence, we define the graph $G=(V,E)$, where $V=\{P^1,\ldots,P^n\}$ and $\{P^i,P^j\}\in E$ iff $P^j\in H(C,P^i)$. To prove the lemma, we just have to show that $G$ is connected. For simplicity, we assume $C=(0,\ldots,0)$, so that $x\in H(C,y)$ precisely when $x\cdot y \leq 0$.

If $G$ is connected, we are done, so we assume that $G$ is not connected, which implies it has a maximally connected subset $A\subsetneq V$. 
The proof that $G$ is connected then uses the following strategy. 
We will first show that $A$ must be contained in some half-space $H(C,Q)$, then we will show that in fact all of $V$ must lie in one such half-space, where $V\setminus A$ cannot lie on the bounding hyperplane. This will then imply that $C\in\Conv(A)$, which, by Lemma~\ref{lem:Seidel} proves that the m.e.b. of $A$ is $\B^d(C,1)$. Finally, using that $\bigcup_{a\in A}H(C,a)=\R^d$ (Lemma~\ref{lem:AllHalfSpacesCoverEverything}), we conclude that each $P^i$ is in some $H(C,a)$, so $G$ is connected.

Assume for the sake of contradiction that there is no $x\in\R^d$ such that $A\subseteq H(C,x)$. Then there is also no $x$ such that $A\cap H(C,x)=\emptyset$, which implies that $A\cap H(C,x)\neq\emptyset$ for all $x\in\R^d$. This holds in particular for all $x\in V\setminus A$, which contradicts the maximal connectedness of $A$ in $G$. Therefore, we may conclude that $A\subseteq H(C,Q)$ for some $Q\in\R^d$.

Now let $m$ be the center of the m.e.b.~of $A$, and assume that $m\neq C$. This implies that the minimum $\min_{x\in \Sp^{d-1}(C,1)}\max_{a\in A} d(x,a)$ is attained at the point $m'$ found by projecting $m$ onto $\Sp^{d-1}(C,1)$ via the line from $C$ to $m$. 
Moreover, the distance between two points on the sphere is monotone decreasing with the inner product of the two points. Therefore, the fact that $Q$ satisfies $\min_{a\in A} -Q\cdot a\geq 0$, implies we also have $\min_{a\in A} m'\cdot a\geq 0$, and we can conclude that $m,m'\in \Cone(A)$ and $A\subseteq H(C,-m)$.

By maximal connectedness of $A$, we know that $A\subseteq \R^d\setminus \bigcup_{b\in V\setminus A}H(C,b)$, which implies $m\in \Cone(A)\setminus\{0\}\subseteq \R^d\setminus \bigcup_{b\in V\setminus A}H(C,b)$. In other words, $m\not\in H(C,b)$ for all $b\in V\setminus A$. This implies that $b\in\R^d\setminus H(C,m)$ for all $b\in V\setminus A$, so we conclude that $V\setminus A \subseteq \R^d\setminus H(C,m)$ and $A\subseteq H(C,-m)$. 

Now note that $C=(0,\ldots, 0)$ must lie in $\Conv(V)$ by Lemma~\ref{lem:Seidel}, and this convex combination $C=\sum_i\lambda_iP^i$ (with $\sum_i\lambda_i=1$ and $0\leq \lambda_i\leq 1$) cannot use any elements from $V\setminus A$, as they all have a nonzero component in the $Cm$ direction, which cannot be compensated for by any other part of the convex combination. More precisely, $a\cdot m \geq 0$ for all $a\in A$ (because $a\in H(C,-m)$), and $b\cdot m > 0$ for all $b\in V\setminus A$ (because $b\not\in H(C,m)$). So, for the equality $0=m\cdot C = \sum_i m\cdot \lambda_iP^i$ to hold, we need $\lambda_i$ to be zero for all $i$ such that $P^i\not\in A$. Hence, $C$ is a convex combination of points in $A$ and $C\in\Conv(A)$. This then  implies that $\B^d(C,1)$ is the m.e.b.~of $A$ (Lemma~\ref{lem:Seidel}). This leads us to the contradition that both $m=C$ and $m\neq C$, so we conclude that the assumption $m\neq C$ is false.

Finally, we have that $m=C$, and we may use $m=C\in \Conv(A)$. Hence, the m.e.b.~of $A$ is equal to the m.e.b.~$\B^d(C,1)$ of $S$ (Lemma~\ref{lem:Seidel}). Then, applying Lemma~\ref{lem:AllHalfSpacesCoverEverything} to $A$ and noting that $A$ is connected in $G$ proves the result.
\end{proof}

\section{Closed Convex Sets}
In this section, we use the connectedness results for simplices to characterize the connectedness of closed convex sets. Clearly, the radius of such a connected set must be at least one, because, otherwise, the center of the m.e.b.~has no neighbours in the induced unit distance graph. We will show that this condition is actually sufficient for a closed convex set to be connected, as long as the set has affine dimension at least 2. We first focus on bounded sets, and consider unbounded sets afterwards.

\begin{lemma}\label{lem:Scale}
Let $X\subseteq \R^d$ be a bounded convex set such that $\G[d][X]$ is connected. Then for all $\lambda\geq 1$, $\G[d][\lambda X]$ is connected as well.
\end{lemma}
\begin{proof}
Without loss of generality, suppose the m.e.b.~of $X$ is $B=\B^d(0,r)$. We prove that $\G[d][\lambda X]$ is connected for all $\lambda\geq 1$ such that $\lambda r\leq r+1$ (i.e., we increase the radius by at most 1). The result then simply follows by repeated application of this result.

Let $\lambda\geq 1$ be such that $\lambda r\leq r+1$, and let $x$ be any point in $\lambda X$. Because $X$ is convex and $0\in X$ we have $X\subseteq \lambda X$. Hence, we only have to prove that for each $x\in \lambda X\setminus X$, there is a point $y\in X$ with $d(x,y)=1$. To prove this, we find two points $P,Q\in X$ at distance $d(P,x)\leq 1$ and $d(Q,x)\geq 1$. The (continuous) line segment $PQ$ is contained in $X$, so, by the intermediate value theorem, there is a point $R\in X$ on $PQ$ such that $d(R,x)=1$.

First we show that we can take $P=x/\lambda$. Because $x\in \lambda X$, we have $x/\lambda\in X$; and because $|x|\leq r(\lambda X)\leq r+1$ and $1-1/\lambda\leq 1/(r+1)$, we have  $d(x,x/\lambda)=|x|(1-1/\lambda)\leq 1$. 
Now we show that there exists a $Q\in X$ with $d(x,Q)\geq 1$. Suppose there weren't such a $Q$, then $X\subseteq \B^d(x,1)$. It then follows that $r(X)< 1$ as $x$ is not the center of the m.e.b. of $X$. However, if $r(X)<1$, then $\G[d][X]$ could not have been connected, which contradicts the conditions of the lemma. Hence, there exists a $Q\in X$ with $d(x,Q)\geq 1$.
\end{proof}

\begin{lemma}\label{lem:EnclosedSimplex}
Let $X\subseteq \R^d$ be a compact and convex set and let $B$ be its m.e.b.~with center $C$, then $C\in \Conv(\del B\cap X)$.
\end{lemma}
\begin{proof}
Suppose, to the contrary, that $C\not\in \Conv(\del B\cap X)$ and consider the sets $X^{\del}(\epsilon)=\Conv(\del B\cap X)\oplus \B^d(0,\epsilon)$ parameterized by $\epsilon>0$. Then there must exist such an $\epsilon>0$ such that $C\not\in \Conv(X^{\del}(\epsilon))$. 
Because $X^{\del}(\epsilon)$ is convex, the hyperplane separation theorem implies the existence of a hyperplane $H$ separating $C$ from $X^{\del}(\epsilon)$. 
Let $L$ be the line segment between $H$ to $C$ perpendicular to $H$.
From $C$ to $H$ along $L$, the distance to each point in $X^{\del}(\epsilon)$ decreases.

Furthermore, for some $\epsilon'>0$, we have the distance bound $d(C,P)<r(X)-\epsilon'$ for all $P\in X\setminus X^{\del}(\epsilon)$. 
This follows from the fact that the closure of $X\setminus X^{\del}(\epsilon)$ is compact and does not contain any points at distance $r(B)$ from $C$ by definition of $X^{\del}(\epsilon)$.
Hence, the triangle inequality implies that for any point $Q\in \B^d(C,\epsilon')$ and any point $P\in X\setminus X^{\del}(\epsilon)$, we have $d(Q,P)\leq d(Q,C)+d(C,P)<\epsilon'+(r(X)-\epsilon')=r(X)$.

Finally, combining these results, there is a point $C'\neq C$ on $L\cap\B^d(C,\epsilon')$ such that $d(C',P)<r(X)$ for all $P\in X$. 
This contradicts the fact that $B$ is the m.e.b. of $X$, so we conclude that $C\in \Conv(\del B\cap X)$.
\end{proof}

\begin{theorem}\label{the:ConvexConnected}
Let $X\subseteq \R^d$ be closed and convex, then $\G[d][X]$ is connected iff $r(X)=0$, or $r(X)\geq 1$ and $X$ contains at least three affinely independent points. Moreover, if $X$ is bounded, then the diameter of $\G[d][X]$ is finite.
\end{theorem}
\begin{proof}
The result is trivial when $X$ contains at most two affinely independent points, so we assume $X$ contains at least three affinely independent points. We first deal with the case where $r(X)$ is unbounded: in that case, $X$ can be covered by a set of compact convex sets with radius at least $1$ which can be reached by taking steps along a ray of $X$. In the remaining case, $r(X)$ is bounded, and we may assume $r(X)=1$ by Lemma~\ref{lem:Scale}.

Let $B$ be the m.e.b.~of $X$ with center $C$, then $C\in \Conv(\del B\cap X)$ because $X$ is closed (Lemma~\ref{lem:EnclosedSimplex}). In particular, there exist at most $d+1$ points $S\subseteq \del B\cap X$ such that $C\in\Conv(S)$ (Caratheodory's theorem). By Lemma~\ref{lem:Seidel}, the m.e.b.~of $S$ is $B$ as well.
As $X$ contains at least three points, $S$ contains at least two points.
If $|S|=2$, then $\Conv(S)$ is a line segment of length two. Furthermore, $X$ contains a third point $P$ not in $\Conv(S)$. Hence, $X$ contains the triangle $T$ formed by $S$ and $P$, which is obtuse of radius one. Hence, $\G[d][T]$ is connected (Lemma~\ref{lem:ObtuseTriangleConnected}), and each pair of points in $\Conv(S)$ is connected by a path in $\G[d][X]$. 
If $|S|>2$, then $\Conv(S)$ is a simplex of dimension at least two. Hence, by Proposition~\ref{prop:SimplexConnected}, $\G[d][\Conv(S)]$ is connected. 

We now know that each point in $\Conv(S)$ can be reached from $C$ in $\G[d][X]$, so in particular, for all $s\in S$, each point on the line segment $Cs$ can be reached from $C$ in $\G[d][X]$. Seeing that $d(s,C)=1$ for all $s\in S$, Lemma~\ref{lem:AllHalfSpacesCoverEverything} implies that $\bigcup_{s\in S}H(C,s)=\R^d$. Hence, for each point $P\in X\subseteq B$, there is an $s\in S$ such that $P\in H(C,s)$. Thus, because $d(C,P)\leq 1$ and $d(s,P)\geq 1$, there is a point $Q$ on $Cs$ such that $d(Q,P)=1$. We noted earlier that all points on $Cs$ can be reached from $C$ in $\G[d][X]$, so $P$ can be reached from $C$ in $\G[d][X]$ as well. As we picked $P\in X$ arbitrarily, each point in $X$ can be reached from $C$ through $\G[d][X]$ and we conclude that $\G[d][X]$ is connected. The finiteness of the diameter of $\G[d][X]$ follows directly from the constructions of the paths in the lemmas.
\end{proof}

By a simple scaling argument, we obtain the following corollary for convex sets. Note that it is incomplete in the sense that connectedness of $\G[d][X]$ for convex sets $X$ is not completely characterized: the case where $r(X)=1$ and $X$ is not closed is still an open problem. We conjecture that the only extra condition for connectedness in that case is that the intersection of $X$ and its m.e.b.~is non-empty. 

\begin{corollary}\label{cor:ConvexNonClosedConnected}
Let $X\subseteq \R^d$ be convex, then $\G[d][X]$ is connected if $r(X)=0$, or $r(X)>1$ and $X$ contains at least three affinely independent points. 
\end{corollary}

\section{Hypercubes and Hyperrectangles} 
Although the connectedness of hyperrectangle-induced unit distance graphs is characterized completely by Theorem~\ref{the:ConvexConnected}, the bound on the graph diameter based on those proofs is quite large. In this section, we will show that the diameter of the unit distance graphs induced by hypercubes is constant in the dimension of the hypercube. For hyperrectangles, we show an upper bound of the diameter that is a linear combination of the dimension and the graph diameter of the `best' two-dimensional rectangle contained in the hyperrectangle.

\begin{proposition}\label{prop:HypercubeDiamBound}
Let $d>1$, then $\G[d][C^d(l)]$ is connected iff $l\geq 2/\sqrt{d}$. Moreover, for $l=2/\sqrt{d}$, the diameter of this graph is at most $8$.
\end{proposition}
\begin{proof}
The connectedness follows directly from Theorem~\ref{the:ConvexConnected} and the fact that $r(C^d(l))\geq 1$ precisely when $l\geq 2/\sqrt{d}$. Hence, we focus on the diameter of hypercubes $C^d(2/\sqrt{d})$. First we show that all points on 1-dimensional faces can be reached from the center point $M=(1/\sqrt{d},\ldots,1/\sqrt{d})$ in at most three steps, then we complete the proof by showing that each other point of the hypercube is at distance one from a point in such a 1-dimensional face. This proves that $\G[d][C^d(2/\sqrt{d})]$ is connected for all $d>1$ with graph diameter at most $8$.

Let $P$ be a point on a 1-dimensional face of the hypercube $C^d(2/\sqrt{d})$. Without loss of generality, assume that $P=(P_1,0,\ldots,0)$,  
then the point $L_P(h)=(P_{1}/2,h,\ldots,h)$ lies within the hypercube and $d(0,L_P(h))=d(P,L_P(h))$ for all $h\in[0,2/\sqrt{d}]$. Moreover, we have the extremal distance bounds $d(0,L_P(0))\leq 1/\sqrt{d}<1$ and $d(0,L_P(2/\sqrt{d}))>\sqrt{(d-1)(2/\sqrt{d})^2}
>1$. By the intermediate value theorem, there is an $h\in[0,2/\sqrt{d}]$ such that $d(0,L_P(h))=d(L_P(h),P)=1$. Therefore, we can reach $P$ from $M$ in at most three steps in $\G[d][C^d(2/\sqrt{d})]$ via the path the points $0$ and $L_P(h)$.

Now, we show that each point in the hypercube is at distance one from a point in a $1$-dimensional face. Let $P'$ be an arbitrary point in the hypercube, and let $VV'$ be any diagonal of the hypercube. Then, without loss of generality, we have $d(P',V)\leq 1$ and $d(P',V')\geq 1$ because each diagonal has length two.
Now consider a path along the 1-dimensional faces of $C^d(2/\sqrt{d})$ between the vertices $V$ and $V'$. By the intermediate value theorem, there is a point $P''$ on this path such that $d(P',P'')=1$. As $P''$ lies on a path through 1-dimensional faces, it lies on a 1-dimensional face, and can thus be reached from $M$ in at most three steps. We conclude that the arbitrary point $P'\in C^d(2/\sqrt{d})$ can be reached from $M$ in at most four steps through $\G[d][C^d(2/\sqrt{d})]$.
\end{proof}

\begin{lemma}\label{lem:RectancleBound}
Let $R^2(l_1,l_2)$ be a $2$-dimensional rectangle with side lengths $l_1\geq l_2$, then $\G[d][R^2(l_1,l_2)]$ is connected iff $\sqrt{l_1^2+l_2^2}\geq 2$. 

Moreover, when $\sqrt{l_1^2+l_2^2}=2$, the diameter of $\G[d][R^2(l_1,l_2)]$ is bounded by $8$ if $l_2\geq1$, and by $4+8 \ceil*{\frac{l_1-1}{2(1-\sqrt{1-l_2^2})}}$ if $l_2<1$.
\end{lemma}
\begin{proof}
The connectedness result follows directly from Theorem~\ref{the:ConvexConnected}, so we focus on the diameter bound. If $l_2\geq 1$, then we can apply exactly the same arguments as in Proposition~\ref{prop:HypercubeDiamBound}, so we continue with the case $l_2<1$. To prove the bound, we show that for any point $P\in R^2(l_1,l_2)$, there is a path from $M=(l_1/2,l_2/2)$ to $P$ in $\G[2][R^2(l_1,l_2)]$ of at most $2+4 \ceil*{\frac{l_1-1}{2(1-\sqrt{1-l_2^2})}}$ steps.

To see this, assume without loss of generality that $P_1\geq l_1/2$ and $P_2\geq l_2/2$. Note that $P$ must then be at distance $1$ from some point on any continuous path from $0$ to $M$ by the intermediate value theorem. In particular, $P$ is at distance 1 from some point on the continuous path between $0$ and $M$ consisting of the line segment $L=[0,l_1-1]\times\{0\}$ (containing the point $0$) and the arc $R^2(l)\cap \Sp^1((l_1,0),1)$ (containing the point $M$), joined in the point $(l_1-1,0)$. Each point on $L$ can be reached in at most $1+4 \ceil*{\frac{l_1-1}{2(1-\sqrt{1-l_2^2})}}$ steps from $M$ (Lemma~\ref{lem:RectangleWiggle}), and each point on the arc in $2$ steps. Hence, $P$ can be reached from $M$ in at most $2+4 \ceil*{\frac{l_1-1}{2(1-\sqrt{1-l_2^2})}}$ steps.
\end{proof}

\begin{proposition}\label{prop:HyperrectangleDiamBound}
Let $R^d(l)$ be a $d$-dimensional rectangle with side lengths $l=(l_1,\ldots,l_d)$, then $\G[d][R^d(l)]$ is connected iff $|l|\geq 2$. 

Moreover, if $|l|= 2$ and $l'_1= \sqrt{\sum_{i\in I}l_i^2}$ and $l'_2=\sqrt{\sum_{i\not\in I}l_i^2}$ for some $\emptyset\neq I\subsetneq [d]$ then the diameter is bounded by $\diam(\G[d][R^d(l)])\leq \diam(\G[2][R^2(l'_1,l'_2)])+2$.
\end{proposition}
\begin{proof}
The connectedness result follows directly from Theorem~\ref{the:ConvexConnected}, so we again focus on the diameter bound. 

Note that for each $\emptyset\neq I\subsetneq [d]$, $R^d(l)$ contains a rectangle with side lengths $l'_1=\sqrt{\sum_{i\in I}l_i^2}$ and $l'_2=\sqrt{\sum_{i\not\in I}l_i^2}$. 
Such a rectangle necessarily contains a diagonal $D$ of $R^d(l)$. Hence, $\G[d][R^2(l'_1,l'_2)]$ is connected, which implies all points on $D$ are connected in $\G[d][R^d(l)]$ by a path of length at most $\diam(\G[d][R^2(l'_1,l'_2)])$. Finally, Lemmas~\ref{lem:AllHalfSpacesCoverEverything} and~\ref{lem:radiusToHalfSpace} imply that each point in $\G[d][R^d(l)]$ is at distance 1 from some point on $D$. Therefore, $\G[d][R^d(l)]$ has diameter at most $\diam(\G[2][R^2(l'_1,l'_2)])+2$.
\end{proof}

For hypercubes, this proposition together with Lemma~\ref{lem:RectancleBound} gives a proof for a diameter bound of $\diam(C^d(2/\sqrt{d}))\leq 10$ independent of Proposition~\ref{prop:HypercubeDiamBound}. This shows that the bound for hyperrectangles in Proposition~\ref{prop:HyperrectangleDiamBound} can possibly still be improved.

\section{Discussion}
In this paper, we have characterized the connectedness of unit distance graphs induced by closed and convex subsets of $\R^d$. For compact convex sets, we have shown that the diameter of this graph is bounded. Moreover, for the smallest connected hypercubes in each dimension, this diameter stays constant with an increasing dimension. Connected components of such graphs (i.e., for strips of the plane) have been considered superficially previously \cite{axenovich2014chromatic,bauslaugh1998tearing}, but never as the main subject of study, as we have done.

There are some obvious problems that we have left open, such as the connectedness of the unit distance graphs induced by non-closed convex sets, or even non-convex sets. The former is partially answered by Corollary~\ref{cor:ConvexNonClosedConnected}, but the edge case of radius one is still open. This question can possibly be answered using techniques similar to the ones we used, but we made ample use of paths through faces of simplices, which may not be present in the non-closed case. The latter question (finding a characterization of the connectedness for non-convex sets) is probably quite hard to answer in general, so it merits restriction to more tractable sets, such as polytopes.

It would be interesting to see whether this translates to computational hardness of such questions as well. For example, we could consider the complexity of the following problem.\\
\\
Input: A (finite) polytope $P\subseteq \R^d$. \\
Question: Is $\G[d][P]$ connected?\\

Interestingly, it is not even immediately clear whether this is in NP; we probably need some kind of characterization of connectedness to verify a certificate. 

Another interesting (computational) problem concerns the diameter: Given a set $X\subset \R^d$ of radius $1$, find a scaling factor $\lambda>0$ such that $\diam(\G[d][\lambda X])$ is minimal. If $\lambda$ is small, the diameter can be very large because of little wiggle room; if $\lambda$ is large, the diameter can be large because of the distance that needs to be bridged using steps of distance one.

Another variation for all these problems is to consider the unit distance graph of $\Q^d$, which is common when studying the chromatic number as well (e.g., \cite{axenovich2014chromatic}). However, only considering connectedness and diameters (next to chromatic number) is quite restrictive, as there are numerous other interesting graph properties. For example, one can study the existence of odd-length cycles. This has been done for strips of the plane in \cite{bauslaugh1998tearing}, where they attempt to find the minimal width $w_n$ of a strip of the plane $S_{w_n}$ so that the graph $\G[2][S_{w_n}]$ contains a cycle of length $n$ (where $n$ is odd).

\subsection{Connected components for small subsets of $\R^d$}
We have focused on induced unit distance graphs that are connected, i.e., that have one connected component. Obviously, we can relax this notion, and instead consider the number and `size' of connected components when $\G[d][X]$ is disconnected.

For example, if the radius of a hyperrectangle $R$ is smaller than $1$, then the midpoint of the hyperrectangle has no neighbours in $\G[d][R]$, but this graph is not necessarily totally disconnected. Of course, this does happen when the radius is smaller than than $1/2$. It is not immediately clear what the components look like for radii between $1/2$ and $1$.

\begin{figure}[h]
    \centering
    \includegraphics[width=\textwidth]{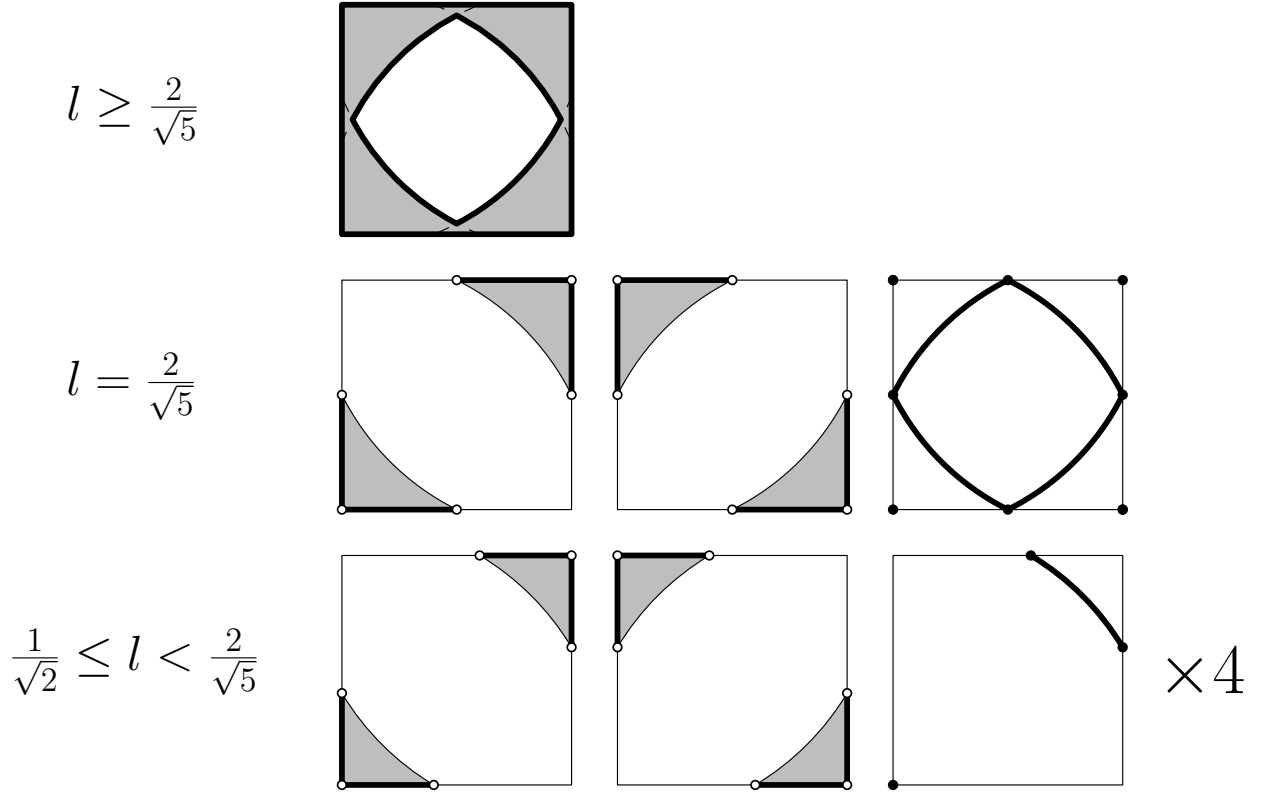}
    \caption{The non-trivial connected components of $\G[2][C^2(l)]$ for varying $l$. In each square, the component consists of the shaded area, the thick black lines, and the filled nodes. Each of the remaining points in the middle of the square, which are in none of the displayed components, is a connected component of its own.}
    \label{fig:componentsSmallSquare}
\end{figure}

A quick illustration of this concept for squares makes it clear that these components can be quite interesting. If the square $C^2(l)$ has side lengths $l> \frac{2}{\sqrt{5}}$, then it seems that $\G[d][C^2(l)]$ has one non-trivial connected component consisting of the points
\[\bigcup_{v\in V}C^2(l)\setminus \mathring{\B}^2(v,1),\]
where $V$ is the set of vertices of $C^2(l)$, and $\mathring{B}^2(v,1)$ is the open ball of radius $1$ around $v$; all points on the interior of the shape bounded by these circle arcs have no neighbours (Figure~\ref{fig:componentsSmallSquare}). If $l=\frac{2}{\sqrt{5}}$, then $\G[d][C^2(l)]$ appears to have 3 non-trivial connected components:
\[
\begin{aligned}
&\left(C^2(l)\setminus \B^2((0,0),1)\cup C^2(l)\setminus \B^2((l,l),1)\right)\setminus\{(0,0),(l,l)\},\\
&\left(C^2(l)\setminus \B^2((0,l),1)\cup C^2(l)\setminus \B^2((l,0),1)\right)\setminus\{(0,l),(l,0)\},\\
&\{(0,0),(0,l),(l,0),(l,l)\}\cup\bigcup_{v,\in V}C^2(l)\cap \Sp^1(v,1),
\end{aligned}
\]
where $\Sp^1(v,1)$ is the circle of radius $1$ around $v$ (Figure~\ref{fig:componentsSmallSquare}). For $1/\sqrt{2}\leq l<\frac{2}{\sqrt{5}}$, the last of these components falls apart into 4 separate components consisting of a vertex $v$ and the arc $C^2(l)\cap \Sp^1(v,1)$.

There is a striking (partial) correspondence between the critical side lengths $l=1/\sqrt{2}$, $l=2/\sqrt{5}$ and $l=\sqrt{2}$ of the square for which the number of non-trivial connected components of $\G[2][C^2(l)]$ changes, and the lengths at which the chromatic number changes \cite{kruskal2008chromatic}---only the critical value $l=8/\sqrt{65}$ for the chromatic number does not show up here. It would be interesting to study the case $l=8/\sqrt{65}$ in more detail, and to investigate whether such a correspondence also holds for hypercubes.

\subsection{Random walks}
Our problem is obviously also related to random walk problems. In most of these problems, the step length has some continuous distribution (often exponential). Here, the step length is $1$ in all cases, and only the new direction is chosen (e.g., uniformly) at random. This leads to the question of determining the distribution for the location after a fixed number of steps. Probabilistic problems of this nature have been investigated thoroughly ever since Pearson asked the following question \cite{pearson1905problem}.
\begin{quote}
A man starts from a point $O$ and walks $l$ yards in a straight line; he then turns through any angle whatever and walks another $l$ yards in a second straight line. He repeats this process $n$ times. I require the probability that after these $n$ stretches he is at a distance between $r$ and $r + \delta r$ from his starting point, $O$.
\end{quote}

However, in related research, the space of the so called \emph{short uniform random walk} is most often unbounded, i.e., $\R^d$ (e.g., \cite{borwein2011some,borwein2016closed,zhou2019borwein}). There is at least some interest in bounded cases---as we have studied in a non probabilistic setting---as well, but then, too the step size may vary, and the boundary may be assumed to be reflexive (e.g., \cite{conolly1987random}).

Although we do not venture into the probabilistic realm in this paper, there are some obvious questions that this paper brings up. For example, we could ask essentially the same question as Pearson for bounded regions, but for a man who is aware of his surroundings and walks stretches of length $1$ without crossing the boundary of the region. 
Note that the starting point is of great influence in this problem, and a uniform distribution on the neighbourhood may either be a continuous one (if there is an infinite number of neighbours), or discrete (if the neighbourhood is finite, such as for the midpoint of $C^2(\sqrt{2})$).

Some examples of such distributions after a small number of steps in squares of varying size are shown in Figure~\ref{fig:Simulations}. These are estimated using $20000$ independent Monte-Carlo simulations of the process.

\begin{figure}[h]
    \centering
    \includegraphics[width=\textwidth]{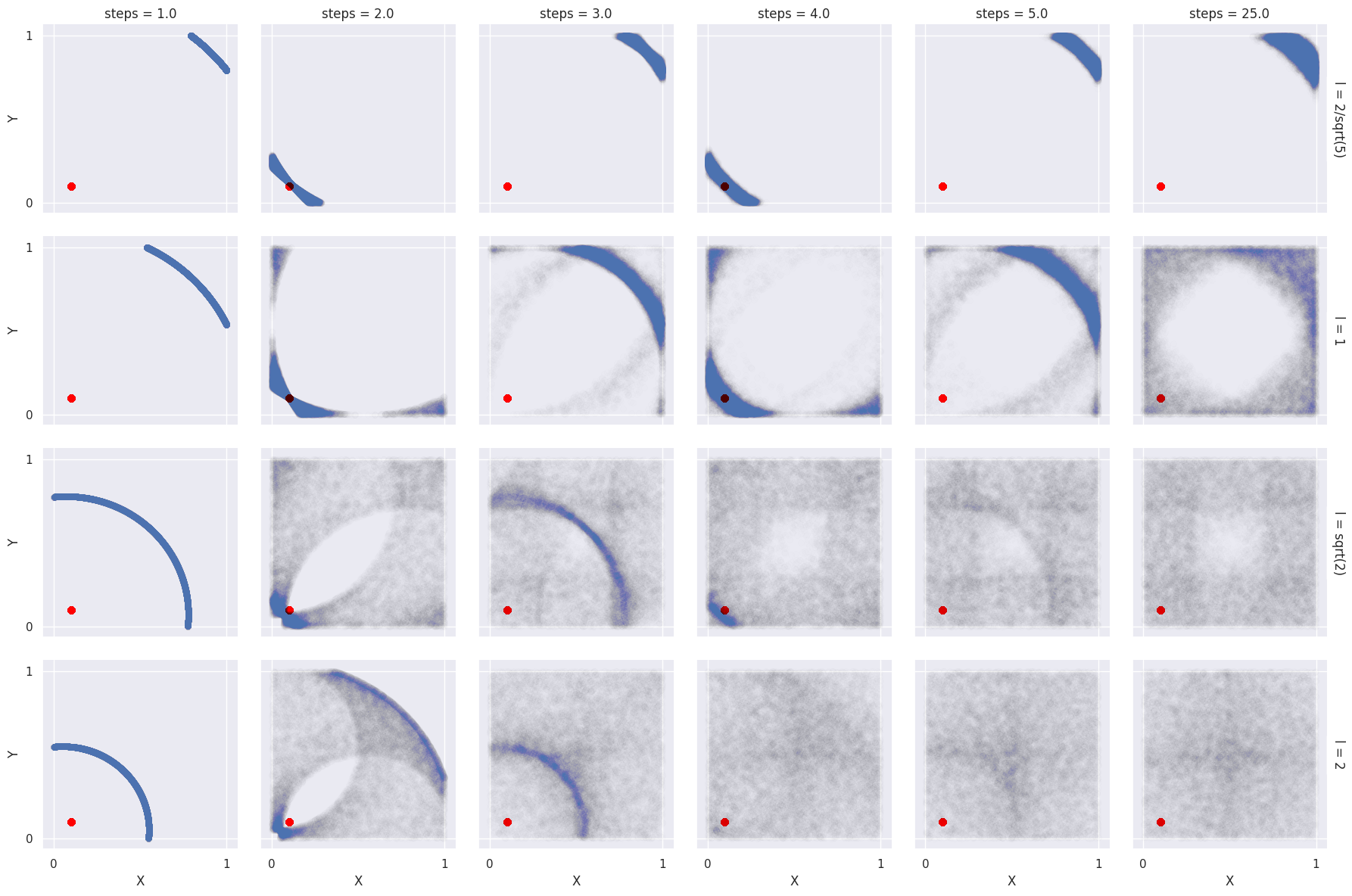}
    \caption{The distribution of the location after a small number of steps (columns, $s=1,2,3,4,5,25$) in different size squares $C^2(l)$ (rows, $l=2/\sqrt{5},1,\sqrt(2),2$). Each distribution is estimated using 20000 independent Monte-Carlo simulations of the process, starting in the point $(0.1,0.1)$ (red). The squares are all scaled to size one, so that the `unit' steps are of different sizes.}
    \label{fig:Simulations}
\end{figure}

As for Pearson's question, we can also inquire about the distribution after an infinite number of steps. i.e., the stationary distribution of the random walk. This distribution may depend on the starting point when the $\G[d][X]$ is disconnected, or is only connected via a node of finite degree.

\end{document}